\documentclass[11pt,a4paper]{article}

\usepackage{amssymb, amsmath}

\def\A{\mathbb{A}}
\def\C{\mathbb{C}}
\def\R{\mathbb{R}}
\def\N{\mathbb{N}}

\def\Q{\mathbb{Q}}
\def\Z{\mathbb{Z}}

\def\cA {{\cal A}}

\def\tr {{\rm tr}}

\newtheorem{defn}{Definition}

\newtheorem{lemma}[defn]{Lemma}
\newtheorem{proposition}[defn]{Proposition}
\newtheorem{theorem}[defn]{Theorem}

\newtheorem{notation}[defn]{Notation}
\newtheorem{example}[defn]{Example}

           {\vspace{3.3mm}
           \noindent{\bf #1}\it}%
           {\vspace{3.3mm}}

\newenvironment{proof}[1]{
  \trivlist \item[\hskip \labelsep{\it #1}]}{\hfill\mbox{$\square$}
  \endtrivlist}

\title{On the minimum of a positive polynomial over the standard simplex}

\author{Gabriela Jeronimo$^{\textrm{a,b,}}$\thanks{Partially supported by the following Argentinian research grants: UBACyT X847 (2006-2010) and PIP CONICET  5852/05.} \quad Daniel Perrucci$^{\textrm{a,}*}$ \\[5mm]
{$^{\textrm{a}}$ {\small Departamento de Matem\'atica, Facultad de
Ciencias
Exactas y Naturales,}}\\
{{\small Universidad de Buenos Aires, Ciudad Universitaria, 1428
Buenos Aires, Argentina}}
\\[2mm]
{$^{\textrm{b}}$ {\small CONICET - Argentina }}}

\begin{document}

\maketitle

\begin{abstract}
We present a new positive lower bound for the minimum value taken by a polynomial $P$ with integer coefficients in $k$ variables over the standard simplex of $\R^k$, assuming that $P$ is positive on the simplex. This bound depends only on the number of variables $k$, the degree $d$ and the bitsize $\tau$ of the coefficients of $P$ and improves all previous bounds for arbitrary polynomials which are positive over the simplex.
\end{abstract}

\section{Introduction}

In the last years, the problem of determining the positivity of a polynomial in $k$ variables with real coefficients in (a subset of) $\R^k$ has been studied extensively with different approaches (see, for instance, \cite{PD}). One of them consists in exhibiting a \emph{certificate of positivity}, that is to say, an algebraic identity showing explicitly that the polynomial is positive over the considered set (see \cite{BCR87}).
In order to construct these certificates of positivity, it is useful to know an \emph{a priori} lower bound for the minimum of a polynomial which only takes positive values on the set (see for instance \cite{PR01}, \cite{Sch04}, \cite{Leroy}). For bounded subsets of $\R^k$, such a bound can be obtained by means of Lojasiewicz inequalities (see \cite{BCR87} or \cite{Sol91}), as it is done in \cite{dLS96} for the case of the standard simplex of $\R^k$. However, these bounds involve a universal constant.

This papers considers the problem of finding an explicit lower bound
for the minimum of a polynomial $P\in \Z[X_1,\dots, X_k]$ over the
standard $k$-dimensional simplex $\Delta_k = \{ x\in \R^k_{ \ge 0}
\mid \sum_{i=1}^k x_i \le 1 \}$, assuming that $P$ takes only
positive values on $\Delta_k$, which depends only on the number of
variables $k$ of $P$, its degree $d$, and an upper bound $\tau$ for
the bitsize of its coefficients.

Under non-degeneracy conditions, a lower bound of this kind can be obtained by applying Canny's gap theorem (\cite{Can87}). In \cite{EMT}, an improved gap theorem is proved and, consequently, a better bound under the same assumptions is derived.
The best known lower bound for the minimum with no extra assumptions on $P$ was given in \cite{BLR}, where the minimum is estimated by means of an analysis of the values that the polynomial takes on the boundary of the simplex and its critical values in the interior.

Here we present a new lower bound for the minimum in the general case which improves the previous ones. Our main result is the following:

\begin{theorem}\label{theorembound}
For every $P\in \Z[X_1,\dots, X_k]$ with degree $d$ and coefficients
of bitsize at most $\tau$ which only takes positive values over the
standard simplex $\Delta_k$, we have
$$\min_{\Delta_k} P \ge 2^{-(\tau+1)  d^{k+1}} d^{-(k+1) d^{k}} \binom{d+k}{k+1}^{-d^k(d-1)}$$ Taking into account that $\binom{d+k}{k+1} \le d^{k+1}$, we obtain the simplified bound
$$\min_{\Delta_k} P \ge 2^{-(\tau+1)  d^{k+1}} d^{-(k+1) d^{k+1}}.$$
\end{theorem}

Our approach combines the application of the critical point method as in \cite{BLR} with deformation techniques similar to those used in \cite{JPS09} to compute critical values. This deformation-based approach enables us to work, even in degenerate cases, with a polynomial system defining the critical points of an associated polynomial instead of taking the sum of squares of the polynomials involved, as it is done in \cite{BLR} leading to an artificial degree growth. Moreover, we estimate the values that the polynomial takes at the critical points by computing upper bounds on the coefficients of the characteristic polynomial of a multiplication map in the associated quotient algebra, with no need of a previous explicit description of these critical points.

\section{A lower bound for the minimum}

For $k\in \N$, consider the $k$-dimensional standard simplex
$$ \Delta_k = \Big\{ x\in \R_{ \ge 0}^k \mid \sum_{i=1}^k x_i \le 1\Big\},$$
and for $k, d, \tau \in \N$, let
$$\cA_{k,d,\tau}= \{ P \in \Z[X_1,\dots, X_k] \mid \deg(P)\le d , \ h(P) \le \tau,\ P(x) >0 \ \forall\, x \in \Delta_k\}$$
(here, $\deg(P)$ denotes the total degree of $P$ and $h(P)$ the
maximum bitsize of its coefficients). We are interested in computing
an explicit lower bound for
$$m_{k,d,\tau} = \min \{ \min_{\Delta_k} P \mid P\in \cA_{k,d,\tau} \},$$
the minimum value over the standard simplex of a polynomial $P\in \cA_{k,d,\tau}$, depending only on $k, d $ and $\tau$.

We will analyze first the case where $P$ attains its minimum only at
interior points of the simplex and then, we will proceed recursively
to deal with the case where the minimum is attained at a point of
the boundary. In order to do this, we consider
\begin{eqnarray*}
\cA_{k,d,\tau}^{(b)} & =& \{ P\in \cA_{k,d,\tau} \mid \exists z \in \partial \Delta_k \hbox{ such that } P(z) = \min_{\Delta_k} P \},\\
\cA_{k,d,\tau}^{(0)} & =& \cA_{k,d,\tau}\setminus \cA_{k,d,\tau}^{(b)}.
\end{eqnarray*}


\subsection{The deformation}\label{deformation}

Fix a polynomial $P\in \cA_{k,d,\tau}^{(0)}$.
Let $Q(X) = \sum_{i=1}^k \frac{1}{d+1} X_i^{d+1}$ and $F(t,X) = P(X) +t Q(X)$. For $i=1,\dots, k$, let $$F_i(t,X) = \frac{\partial F}{\partial X_i} =
\frac{\partial P}{\partial X_i} +t X_i^d.$$

Following \cite{JPS09}, consider the variety $\widehat V =
V(F_1,\dots, F_k) \subseteq \A_\C^1 \times \A_\C^k$ and its
decomposition
$$\widehat V  = V^{(0)} \cup V^{(1)}\cup V,$$
where $ V^{(0)}$ is the union of the irreducible components of $\widehat V$ contained in $\{ t= 0\}$, $ V^{(1)}$ is the union of the irreducible components of $\widehat V$ contained in $\{ t= t_0\}$ for some $t_0\in \C\setminus \{0\}$ and $V$ is the union of the remaining irreducible components of $\widehat V$.

\begin{lemma}\label{minimizer}
There exists $z_0\in \Delta_k$ such that $P(z_0) = \min_{\Delta_k}
P$ and $(0,z_0) \in V$.
\end{lemma}

\begin{proof}{Proof.}
Let $\varepsilon>0$ such that $\varepsilon< |t_0|$ for every $t_0\in
\pi_t(V^{(1)})$ (here, $\pi_t: \A^1\times \A^n\to \A^1$ denotes the
projection to the first coordinate $t$). Let $(t_n)_{n\in \N}$ be a
decreasing sequence of positive real numbers with $t_1 <\varepsilon$
and $\lim_{n\to \infty} t_n = 0$. For every $n\in \N$, let $w_n \in
\Delta_k$ such that $F(t_n, w_n) = \min_{ \Delta_k} F(t_n, - )$. We
may assume that the sequence $(w_n)_{n\in \N}$ converges to a point
$ z_0\in \Delta_k$. Therefore, for every $z\in \Delta_k$, we have
$$P(z_0) = F(0,z_0) = \lim_{n\to \infty} F(t_n, w_n)\le \lim_{n\to \infty} F(t_n, z) = F(0,z) = P(z).$$

We conclude that  $P(z_0)= \min_{\Delta_k} P$. As $P\in
\cA_{k,d,\tau}^{(0)}$, the point $z_0$ lies in the interior $
\Delta_k^\circ$ of $\Delta_k$ and, therefore, $w_n\in
\Delta_k^\circ$ for $n\gg 0$.  Then, for every $n\gg 0$, $w_n$ is a
local minimum of $F(t_n, -)$ and so, $F_i(t_n, w_n)=0$ for
$i=1,\dots, k$. By the choice of $t_1$, it follows that $(t_n,
w_n)\in V$ for $n\gg0$ and therefore, $(0,z_0)\in V$.
\end{proof}

Assume $P = \sum_{|\alpha|\le d} a_\alpha X^\alpha$ and let $R\in
\Z[X]$ be the polynomial $$R(X) = d\cdot P(X) - \sum_{1 \le i \le
k}X_i\frac{\partial P}{\partial X_i}(X) = \sum_{|\alpha|\le d-1} (d
- |\alpha|)a_\alpha X^{\alpha}.$$ Note that for a point $z_0$ as in
Lemma \ref{minimizer}, since $\frac{\partial P}{\partial X_i}(z_0) =
0$ for $1 \le i \le k$, we have that $R(z_0) = d\cdot P(z_0)$.

Let $W = \C(t)[X_1,\dots, X_k]/(F_1,\dots, F_k)$, which is a  $\C(t)$-vector space of dimension $d^k$. Moreover, if
$$U = \{\gamma = (\gamma_1,\dots, \gamma_k)\in \Z^k \mid 0\le \gamma_i \le d-1\hbox{ for every } 1\le i\le k\},$$
we have that $\{ X^\gamma \mid \gamma\in U\}$ is a basis of $W$.
For a polynomial $g\in \Z[X]$, $m_g$ will denote
 the multiplication map $m_g: W\to W$, $m_g([f])= [g\cdot f]$,  and $\chi({m_g})\in \C(t)[Y]$ the characteristic polynomial of this linear map.

We are going to show that $\chi({m_R}) (t, Y) = S(t,Y) /t^l$, where
$S(t,Y)\in \Q[t,Y]$, $l\in \N_0$, and $S(0,Y)\not\equiv 0$. Then,
since $\chi({m_R})(t,R(X)) \in (F_1,\dots, F_k) \C(t)[X_1,\dots, X_k]$,
we have that there is a polynomial $\alpha(t) \in \C[t]$ such that
$\alpha(t) S(t, R(X)) \in (F_1,\dots, F_k) \C[t,X_1,\dots, X_k]$.
Therefore, $S(t,R(X) )\in I(V)$ and so, $S(0, R(z_0)) = 0$. The bound
on the minimum of the polynomial $P$ over the standard simplex will
be obtained from upper bounds on the size of the  coefficients of
$S(0,Y)$.

\subsection{Estimates for computations in the quotient algebra}

In order to analyze the characteristic polynomial $\chi(m_R)$, we
start by studying re-writing techniques in the basis $\{X^\gamma \mid
\gamma \in U\}$ of $W$. We follow the approach in \cite[Chapter
12]{BPR}.

For every $\beta \in \N_0^k$, the residue class of the monomial
$X^\beta$ in $W$ can be written in the form $X^\beta=
\sum_{\gamma\in U} x_{\beta, \gamma} X^\gamma$ for some elements
$x_{\beta, \gamma}\in \C(t)$. Moreover, we have:

\begin{lemma}\label{pols1t}
For every $\beta \in \N_0^k$ and every $\gamma \in U$, there is a univariate polynomial $c_{\beta, \gamma}\in \Z[T]$ such that $x_{\beta, \gamma} = c_{\beta, \gamma}(\frac{1}{t})$. Moreover, if $\beta \notin U$,  $c_{\beta, \gamma} = 0$ for every $\gamma $ with $|\gamma|\ge |\beta|$.
\end{lemma}

\begin{proof}{Proof.}
First note that, for $\beta\in U$, the identity holds trivially with $c_{\beta, \gamma} = 0$ if $\gamma \ne \beta$ and $c_ {\beta, \gamma} = 1$ if $\gamma = \beta$.

For $\beta \notin U$, there exists an index $i$ such that $\beta_i
\ge d$ and, so, $\beta = \tilde \beta + d e_i$ with $\tilde \beta
\in \N_0^k$. We proceed by induction on $|\beta|$, starting with
$|\beta| = d$. In this case, we have that $\beta = d e_i$ and the
following identity holds in $W$:
\begin{equation}\label{rewriting}
 X_i^d =  \sum_{|\alpha| <d} - a_{\alpha +e_i} (\alpha_i +1) \frac1t X^{\alpha }.
 \end{equation}
We conclude that $c_{\beta, \gamma}=  - a_{\gamma +e_i} (\gamma_i +1) T$ if $|\gamma|< d = |\beta|$ and $c_{\beta, \gamma}= 0$ if $|\gamma|\ge d = |\beta|$.

Now, if $\beta = \tilde \beta + d e_i$, we have
$$X^\beta = X_i^d X^{\tilde \beta} =
\sum_{|\alpha| <d} - a_{\alpha +e_i} (\alpha_i +1) \frac1t X^{\alpha +\tilde
\beta};$$
therefore,
$$ X^\beta=\sum_{|\alpha| <d,\, \alpha +\tilde \beta \in U} - a_{\alpha +e_i} (\alpha_i +1) \frac1t X^{\alpha +\tilde \beta}
+ \sum_{|\alpha| <d,\, \alpha +\tilde \beta \notin U} - a_{\alpha
+e_i} (\alpha_i +1) \frac1t \sum_{\gamma\in U} x_{\alpha +\tilde
\beta, \gamma} \, X^{\gamma}$$
\begin{equation}\label{rewriting2}
 = \sum_{\gamma \in U, \,\gamma = \alpha +\tilde \beta,\, |\alpha|<d} -a_{\alpha +e_i} (\alpha_i +1) \frac1t X^{\gamma} + \sum_{\gamma \in U}
\Big(\sum_{|\alpha |< d,\, \alpha +\tilde \beta \notin U }
-a_{\alpha+e_i}(\alpha_i+1) \frac1t \, x_{\alpha +\tilde \beta,
\gamma}\Big) X^\gamma.
\end{equation}
Note that for every $\alpha$ such that $|\alpha|<d$, we have that
$|\alpha + \tilde \beta| =|\alpha |+ |\tilde \beta|<d +|\beta| -d =
|\beta|$; then,  by our inductive assumption, it follows that
$x_{\alpha +\tilde \beta, \gamma} = 0$ whenever $\alpha +\tilde
\beta \notin U$ and $|\alpha + \tilde \beta| \le |\gamma|$. Using
the previous identity, this implies that $x_{\beta , \gamma} = 0$
for every $\gamma \in U$ with $|\gamma |\ge |\beta|$.

The inductive assumption also states that $x_{\alpha +\tilde \beta, \gamma}= c_{\alpha +\tilde \beta, \gamma}(\frac{1}{t})$ for every $\alpha$ with $|\alpha|<d$ and every $\gamma \in U$; therefore, taking into account identity (\ref{rewriting2}), for every $\gamma \in U$ with $|\gamma| < |\beta|$, we have that $x_{\beta, \gamma } = c_{\beta, \gamma}(\frac1t)$, where
 \begin{equation}\label{coeffs1}
 c_{\beta, \gamma} = \sum\limits_{|\alpha|<d,\, \alpha +\tilde \beta \notin U} - a_{\alpha +e_i} (\alpha_i +1) c_{\alpha +\tilde \beta, \gamma} T\in \Z[T]
 \end{equation}
 if   $\gamma \ne \alpha +\tilde \beta$  for every $ \alpha$  with $|\alpha| <d$, and
\begin{equation}\label{coeffs2}
c_{\beta, \gamma} = - a_{(\tilde \alpha + e_i)} (\tilde \alpha_i +1)  T + \sum\limits_{|\alpha|<d,\, \alpha +\tilde \beta \notin U} - a_{\alpha +e_i} (\alpha_i +1) c_{\alpha +\tilde \beta, \gamma }\, T\in \Z[T]
 \end{equation}
 if $ \gamma = \tilde \alpha +\tilde \beta$ with $ |\tilde \alpha| <d$.
\end{proof}

\begin{notation}
For a univariate polynomial $c \in \Z[T]$, we use the notation $c_l$ to indicate the coefficient of the monomial $T^l$ in $c$.
\end{notation}

\begin{lemma}\label{bound_c}
For every $\beta\in \N_0^k- U$ and every $\gamma \in U$ with $|\gamma|<|\beta|$, $\deg c_{\beta, \gamma}\le |\beta| - |\gamma|$ and, for $0 \le l \le |\beta| - |\gamma|$, $$|c_{\beta, \gamma, l}| \le 2^{l\tau}d\binom{d+k}{k+1}^{l-1}.$$
\end{lemma}

\begin{proof}{Proof.} The proof is done by induction on $|\beta|$.
If $|\beta| =d$, then $\beta = de_i$ for some index $i$ with $1\le i \le k$ and so, we have that either $c_{\beta, \gamma} =0$ or
$c_{\beta, \gamma} = -a_{\gamma + e_i}(\gamma_i +1)T$ (see identity (\ref{rewriting})). In any case, the result holds.

Suppose now that $|\beta|  >d$. There exists an index $i$ such that
$\beta = \tilde \beta + d e_i$ with $\tilde \beta \in \N_0^k$.
By the inductive hypothesis, for every  $|\alpha|<d$ with $ \alpha +\tilde \beta \notin U$,
$$\deg  c_{\alpha +\tilde \beta, \gamma }\, T \le |\alpha + \tilde \beta|-|\gamma| + 1 \le |\beta|-|\gamma|;$$
so, identities (\ref{coeffs1}) and (\ref{coeffs2}) imply that the stated degree bound for $c_{\beta, \gamma}$ holds.

Note that $c_{\beta, \gamma, 0} = 0$, and $c_{\beta, \gamma, 1} = -({\tilde \alpha}_i + 1)a_{\tilde \alpha + e_i} $
if there exists $\tilde \alpha\in \N_0^k$ with
$|\tilde \alpha| < d$ and $\gamma = \tilde \alpha + \tilde \beta$, and  $c_{\beta, \gamma, 1} = 0$ otherwise. In any case, the bound on the coefficient size holds for $l = 0,1$. Consider now the case $l \ge 2$; from identities (\ref{coeffs1}) and (\ref{coeffs2}), using the inductive assumption we have
\begin{eqnarray*}
|c_{\beta, \gamma, l}| &=&
\Big|\sum\limits_{|\alpha|<d,\, \alpha +\tilde \beta \notin U} - a_{\alpha +e_i} (\alpha_i +1) \, c_{\alpha +\tilde \beta, \gamma, l-1 }\Big| \\
&\le&
\sum\limits_{|\alpha|<d,\, \alpha +\tilde \beta \notin U} 2^\tau(\alpha_i +1) 2^{(l-1)\tau}d\binom{d+k}{k+1}^{l-2} \\
& =& 2^{l\tau}d\binom{d+k}{k+1}^{l-2} \sum_{0 \le e \le d-1} \sum\limits_{|\alpha|<d,\, \alpha +\tilde \beta \notin U, \, \alpha_i = e } (e +1)\\
&\le& 2^{l\tau}d\binom{d+k}{k+1}^{l-2} \sum_{0 \le e \le d-1} (e+1)\binom{d-1-e+k-1}{k-1}.
\end{eqnarray*}
The result follows noticing that
$$\sum_{0 \le e \le d-1} (e+1)\binom{d-1-e+k-1}{k-1} =
\sum_{0 \le e \le d-1} \sum_{0 \le j \le e}\binom{d-1-e+k-1}{k-1} =$$
$$= \sum_{0 \le j \le d-1} \sum_{j \le e \le d-1} \binom{d-1-e+k-1}{k-1} =
\sum_{0 \le j \le d-1} \binom{d-1+k-j}{k} = \binom{d+k}{k+1}.$$
\end{proof}

\subsection{Bounds for traces and characteristic polynomial coefficients}

To estimate the size of the coefficients of the characteristic
polynomial $\chi(m_R)\in \Z[\frac1t][Y]$, we will use the following
relationship with the traces of the multiplication maps by the
powers of $R$ (see for instance \cite[Chapter 12]{BPR}): if
 $\chi(m_R)(Y) = \sum_{h=0}^{d^k}b_{d^k - h} Y^{d^k - h}$, we have
\begin{itemize}
\item $b_{d^k} = 1$,
\item for $1 \le h \le d^k$,
\begin{equation}\label{recurscoeffs}
b_{d^k - h} = - \frac1h \sum_{n = 1}^h \tr(m_{R^n})b_{d^{k}-h+n}.
\end{equation}
\end{itemize}
(See also \cite{GVT95} or \cite{Rou99}, where this technique has
been used for this task and, more generally, for the computation of
a rational univariate representation of the solutions to a
zero-dimensional polynomial system.)

For $n \in \N$, let $R^n(X) := \sum_{|\alpha| \le (d-1)n}
R^{(n)}_{\alpha}X^{\alpha}$. Let us observe that
\begin{eqnarray*}
\sum_{|\alpha| \le d-1} |R^{(1)}_\alpha| &\le& \sum_{|\alpha| \le d-1} (d-|\alpha|)|a_\alpha| \le 2^\tau \sum_{0\le e \le d-1}(d-e)\binom{e+k-1}{k-1} \\
 &=& 2^\tau \sum_{0\le e' \le d-1}(e'+1)\binom{d-1 -e'+k-1}{k-1} = 2^\tau
 \binom{d+k}{k+1},
 \end{eqnarray*}
where the last identity was shown in the proof of Lemma
\ref{bound_c};  in the general case,
\begin{equation}\label{powersize}
\sum_{|\alpha| \le (d-1)n} |R^{(n)}_\alpha| \le \Big(\sum_{|\alpha|
\le d-1} |R^{(1)}_\alpha|\Big)^n \le
\Big(2^\tau\binom{d+k}{k+1}\Big)^n.
\end{equation}

In the sequel, for every $n \in \N$, we will use the same notation
$m_{R^n}$ to denote the multiplication map by $R^n$ in $W$ or the
matrix of this linear map in the basis $\{X^\gamma \mid \gamma \in
U\}$. Rows and columns of these matrices will be indexed by the
exponent vectors $\gamma \in U$.

{}From Lemma \ref{pols1t} and the fact that $R\in \Z[X]$, it follows
that the entries of the matrices $m_{R^n}$ are polynomials in
$\Z[\frac1t]$ and, therefore, the same holds for their traces.

\begin{lemma}\label{traces} For every $n \in \N$, $\deg_{\frac1t} \tr(m_{R^n}) \le n(d-1)$ and, for $0 \le l \le n(d-1)$,
$$|\tr(m_{R^n})_l| \le 2^{(l + n)\tau}d^{k+1}\binom{d+k}{k+1}^{l+n-1}.$$
\end{lemma}

\begin{proof}{Proof.}  For every $n \in \N$ and $\gamma \in U$,
$(m_{R^n})_{\gamma, \gamma} =
\sum_{|\alpha| \le n(d-1)} R^{(n)}_\alpha c_{\gamma + \alpha, \gamma}\big(\frac{1}t\big)$, where
 $c_{\gamma + \alpha, \gamma}$ is a constant if $\gamma + \alpha \in U$ and $\deg c_{\gamma + \alpha, \gamma} \le |\gamma + \alpha| - |\gamma| = |\alpha| \le n(d-1)$ if $\gamma + \alpha \not\in U$. Now,
$$|(m_{R^n})_{\gamma, \gamma, 0}| \le
\sum_{|\alpha| \le n(d-1)} |R^{(n)}_\alpha c_{\gamma + \alpha, \gamma, 0}|
= |R^{(n)}_0| = |R^{(1)}_0|^n \le 2^{n\tau}d^n \le 2^{n\tau}d\binom{d+k}{k+1}^{n-1},$$
and, for $1 \le l \le n(d-1)$,
\begin{eqnarray*}
|(m_{R^n})_{\gamma, \gamma, l}| &\le&
\sum_{|\alpha| \le n(d-1)} |R^{(n)}_\alpha c_{\gamma + \alpha, \gamma, l}|
= \sum_{|\alpha| \le n(d-1), \, \gamma + \alpha \not \in U} |R^{(n)}_\alpha c_{\gamma + \alpha, \gamma, l}| \\
&\le& \sum_{|\alpha| \le n(d-1)} |R^{(n)}_\alpha| \,
2^{l\tau}d\binom{d+k}{k+1}^{l-1} \le\
2^{(l+n)\tau}d\binom{d+k}{k+1}^{l+n-1},
\end{eqnarray*}
where the last inequality follows from (\ref{powersize}). The stated
inequalities are now a consequence of the fact that the dimension of
$W$ is $d^k$.
\end{proof}

We are now ready to find upper bounds for the size of the coefficients of the characteristic polynomial $\chi(m_R)\in \Z[\frac1t][Y]$.

\begin{lemma}\label{bound_b} For $0 \le h \le d^k$, $\deg_{\frac1t} b_{d^k - h} \le h(d-1)$ and, for $0 \le l \le h(d-1)$,
$$|b_{d^k -h, l}| \le 2^{(l+h)(\tau+1)}d^{(k+1)h}\binom{d+k}{k+1}^{l}.$$ The last inequalities are strict for $h \ge 1$.
\end{lemma}

\begin{proof}{Proof.} Let us prove first the degree bound.
The proof is done by induction on $h$ and using the recursive
formula (\ref{recurscoeffs}) for the coefficients $b_{d^k - h}$. For
$h = 0$, the result holds. Now, for $h > 0$, for every $1 \le n \le
h$, by Lemma \ref{traces} and the inductive assumption, $\deg
(\tr(m_{R^n})b_{d^{k}-h+n}) \le n(d-1) + (h-n)(d-1) = h(d-1)$;
therefore, $\deg b_{d^{k}-h} \le h(d-1)$.

Now we prove the bound on the size of the coefficients. For $h = 0$, the result is clear. For $h \ge 1$,
\begin{eqnarray*}
|b_{d^k - h, l}| &=& \frac1h \ \Big|\sum_{n = 1}^h \
\sum_{\genfrac{}{}{0pt}{}{l_1 + l_2 = l}{
\genfrac{}{}{0pt}{}{0 \le l_1 \le n(d-1)}{0 \le l_2 \le (h-n)(d-1)}}} \tr(m_{R^n})_{l_1}b_{d^{k}-h+n,l_2}\Big| \\
&\le&  \frac1h \sum_{n = 1}^h  \sum_{\genfrac{}{}{0pt}{}{l_1 + l_2 = l}{
\genfrac{}{}{0pt}{}{0 \le l_1 \le n(d-1)}{0 \le l_2 \le (h-n)(d-1)}}}
2^{(l_1+n)\tau}d^{k+1}\binom{d+k}{k+1}^{l_1 + n - 1}2^{(l_2 + h-n)(\tau + 1)}d^{(k+1)(h-n)}\binom{d+k}{k+1}^{l_2}\\
& = & 2^{(l +h)\tau}d^{k+1}\binom{d+k}{k+1}^{l  - 1}\frac1h \sum_{n = 1}^h
\binom{d+k}{k+1}^n2^{h-n}d^{(k+1)(h-n)}
 \sum_{\genfrac{}{}{0pt}{}{l_1 + l_2 = l}{
\genfrac{}{}{0pt}{}{0 \le l_1 \le n(d-1)}{0 \le l_2 \le (h-n)(d-1)}}}  2^{l_2} \\
 &<& 2^{(l+h)\tau }d^{k+1}\binom{d+k}{k+1}^{l - 1} \binom{d+k}{k+1}2^{h-1}d^{(k+1)(h-1)}\ 2^{l+1}\\
 &=& 2^{(l+h)(\tau + 1)}d^{(k+1)h}\binom{d+k}{k+1}^{l}.
\end{eqnarray*}
\end{proof}

\subsection{Obtaining the bound}

As explained in Subsection \ref{deformation}, from the characteristic polynomial $\chi(m_R)$,  we can obtain a univariate polynomial having $R(z_0)$ as one of its roots; thus, we get a lower bound for this value in terms of the size of the coefficients of this polynomial.

\begin{proposition} \label{boundminimizer}
Let $z_0$ be as in Lemma \ref{minimizer}. Then,
$$\frac{1}{P(z_0)} \le 2^{d^{k+1}(\tau + 1)}d^{(k+1)d^k}\binom{d+k}{k+1}^{d^k(d-1)}.$$
\end{proposition}

\begin{proof}{Proof.} Take $l_0 := \max_{0 \le h \le d^k} \deg b_{d^k - h}$.
Then,  $\chi(m_R)= \frac{S(t,Y)}{t^{l_0}}$, where  $$S(t, Y) = t^{l_0} \sum_{h = 0}^{d^k} b_{d^k - h}(\frac1t)Y^{d^k - h} =
 \sum_{h = 0}^{d^k} \sum_{l = 0}^{l_0} b_{d^k - h,l}\, t^{l_0 - l}\, Y^{d^k - h}\in \Z[t,Y],$$
and, therefore,
$S(0, Y) =
 \sum_{h = 0}^{d^k} b_{d^k - h,l_0} Y^{d^k - h}\in \Z[Y].$
Since $(0, z_0) \in V$, we have that $S(0, R(z_0)) = 0$, which implies that $\frac{1}{R(z_0)}$ is a root of the polynomial $\sum_{h = 0}^{d^k} b_{d^k - h,l_0} Y^h .$

If $l_0 > (d^k-1)(d-1)$, then $b_{d^k - h, l_0} = 0$ for every $0 \le h \le d^k -1$, and so $b_{0, l_0}\big(\frac1{R(z_0)}\big)^{d^k} = 0$, which is impossible since both factors are nonzero.
Let $h_1 := \max\{h \ | \ b_{d^k - h, l_0} \ne 0\} \le d^k$. By \cite[Prop.~2.5.9]{MS},
$$\frac{1}{R(z_0)} \le \max_{0 \le h \le h_1-1}\Big|\frac{b_{d^k - h, l_0}}{b_{d^k - h_1, l_0}}\Big| + 1.$$
Since $b_{d^k - h_1, l_0} \in \Z-\{0\}$ and the size inequalities in Lemma \ref{bound_b} are strict for $h>0$,
$$\frac{1}{d\cdot P(z_0)} = \frac{1}{R(z_0)} \le 2^{(d^k-1)d(\tau + 1)}d^{(k+1)(d^k-1)}\binom{d+k}{k+1}^{(d^k-1)(d-1)},$$
which implies the result.
\end{proof}

By Lemma \ref{minimizer} and  Proposition \ref{boundminimizer},
we deduce the following lower bound for the minimum of a positive polynomial over the standard simplex in the case this minimum is attained only at interior points of the simplex:

\begin{proposition}\label{minint}
Let $P\in \cA_{k,d,\tau}^{(0)}$. Then $$\min_{\Delta_k} P \ge
2^{-(\tau+1)  d^{k+1}} d^{-(k+1) d^{k}}
\binom{d+k}{k+1}^{-d^k(d-1)}.$$
\end{proposition}

\subsection{Proof of the main result}

The case where the minimum is attained at a point of the boundary of $\Delta_k$ can be dealt with recursively, since the facets of $\Delta_k$ are standard $(k-1)$-dimensional simplices.

We are now ready to prove the main result of the paper.


\begin{proof}{Proof of Theorem \ref{theorembound}.}
We argue by induction on $k$. For $k=1$, the bound is a consequence of Proposition \ref{minint} and the fact that $P(0)\ge 1$ and $P(1)\ge 1$ for every $P\in \cA_{k, d, \tau}$.

Assume now $k>1$ and let $P\in \cA_{k,d,\tau}$. When $d=1$, $P$ is a linear affine polynomial and so, the minimum is attained at a vertex of the simplex, which implies that it is an integer. Then, $m_{k,1,\tau} \ge 1$ for every $k, \tau$. Thus, we may assume $d\ge 2$.

If $P\in \cA_{k,d,\tau}^{(0)}$, the bound follows from Proposition
\ref{minint}. Suppose $P\in \cA_{k,d,\tau}^{(b)}$ and let $z\in
\partial \Delta_k$ with $P(z) = \min_{\Delta_k} P$. If $z_i = 0$ for some
$1\le i\le k$, the polynomial $P_i$ obtained by evaluating  $X_i = 0$ in $P$ satisfies $P_i \in \cA_{k-1,d, \tau}$ and
$$P(z) = P_i(z_1,\dots, \widehat z_i,\dots, z_n) \ge m_{k-1, d, \tau}\ge 2^{-(\tau+1)
d^{k}} d^{-k d^{k-1}} \binom{d+k-1}{k}^{-d^{k-1}(d-1)}$$
(here, $(z_1,\dots, \widehat z_i, \dots, z_n)\in \Delta_{k-1}$ is the point obtained by removing the $i$th coordinate from $z\in \Delta_k$).
On the other hand, if
$\sum_{i=1}^k z_i =1$, consider the polynomial $\widetilde P =
P(X_1,\dots, X_{k-1}, 1-(X_1 +\cdots +X_{k-1}))$. By \cite[Lemma
2.3]{BLR}, $\widetilde P \in \cA_{k-1,d, \tau +1+ d \log k}$ and,
therefore $$P(z) = \widetilde P(z_1,\dots, z_{k-1}) \ge m_{k-1, d, \tau +1+ d \log k}\ge  2^{-(\tau +2+ d
\log k) d^{k}} d^{-k d^{k-1}}\binom{d+k-1}{k}^{-d^{k-1}(d-1)}.$$

In order to finish the proof, it suffices to show that for every $d \ge 2$, and every $k\in \N$,
\begin{equation}\label{induction}
2^{d^{k}(\tau + 2 + d \log k)}d^{kd^{k-1}}\binom{d+k-1}{k}^{d^{k-1}(d-1)}
\le 2^{d^{k+1}(\tau + 1)}d^{(k+1)d^k}\binom{d+k}{k+1}^{d^k(d-1)}.
\end{equation}
First, we show by induction on $d$, that for every $d\ge 2$ and every $k\in \N$, the inequality
$k^{2d-1} \le 2^{d^2 - 2d}d^{(k+1)d -k}$ holds: the case $d=2$ follows easily; in addition,
\begin{eqnarray*}
k^{2(d+1)-1} &\le&  2^{k+4}k^{2d-1} \le 2^{2d-1}d^{k+1}k^{2d-1} \le
  2^{2d-1}d^{k+1}2^{d^2 - 2d}d^{(k+1)d -k}\\ & = & 2^{(d+1)^2 - 2(d+1)}d^{(k+1)(d+1) -k}  \le 2^{(d+1)^2 - 2(d+1)}(d+1)^{(k+1)(d+1) -k}.
  \end{eqnarray*}
Then,
$k^{2d^k-d^{k-1}} = (k^{2d-1})^{d^{k-1}} \le (2^{d^2 - 2d}d^{(k+1)d -k})^{d^{k-1}} = 2^{d^{k+1} - 2d^k}d^{(k+1)d^k -kd^{k-1}}$
and, therefore,
$$2^{2d^k}k^{d^{k+1}}d^{kd^{k-1}}\binom{d+k-1}{k}^{d^{k-1}(d-1)} \le
2^{d^{k+1}}d^{(k+1)d^{k}}k^{d^{k+1}-2d^k+d^{k-1}}\binom{d+k-1}{k}^{d^{k-1}(d-1)}.$$
 Since $\binom{d+k}{k+1}\ge k$ and $\binom{d+k}{k+1}\ge \binom{d+k-1}{k}$, we conclude that
$$2^{2d^k}k^{d^{k+1}}d^{kd^{k-1}}\binom{d+k-1}{k}^{d^{k-1}(d-1)} \le
2^{d^{k+1}}d^{(k+1)d^{k}}\binom{d+k}{k+1}^{d^{k}(d-1)},$$
which implies that inequality (\ref{induction}) holds.
\end{proof}

\section{An example}

The following example shows that the doubly exponential character of the bound is unavoidable.

\begin{example}
Let $\tau$ and $d$ be even positive integers, $d \ge 4$. Consider the polynomial
$$P(X_1, \dots, X_k) = (2^{\tau/2} X_1 - 1)^2 + (X_2 - X_1^{d/2})^2 + \dots + (X_{k} - X_{k-1}^{d/2})^2 + X_k^d.$$
Note that $P$ is positive over $\R^k$. 
Substituting $X_i = 2^{-\frac{\tau}{2}(\frac{d}{2})^{i-1}}$
for $i = 1, \dots, k$, it follows that the minimum of $P$ over the standard simplex of $\R^k$ is lower than or equal to
$2^{-\tau(\frac{d}{2})^{k}}$.
\end{example}

\bigskip

\noindent \textbf{Acknowledgements.} The authors wish to thank
Marie-Fran\c{c}oise Roy for suggesting them the problem and a
possible approach to it. They also acknowledge Saugata Basu, Richard
Leroy and Marie-Fran\c{c}oise Roy for their comments on a first
version of the paper.

\end{document}